\documentclass[sigconf, nonacm, screen]{acmart} % Archive version.
\usepackage{mathtools, xcolor,bm}
\usepackage{amsmath}
\usepackage{amsthm}
\usepackage[utf8]{inputenc}
\usepackage{graphicx}
\usepackage[capitalize]{cleveref}
\crefname{equation}{}{}
\crefname{item}{}{}

\usepackage{url}

\definecolor{blauw}{RGB}{61,158,255}
\definecolor{donkerblauw}{RGB}{0,0,255}
\definecolor{donkergroen}{RGB}{46,148,0}
\definecolor{donkerrood}{RGB}{204,0,0}
\definecolor{donkergroen2}{RGB}{0,95,0}

\newcommand{\N}{\mathbb{N}}
\newcommand{\Z}{\mathbb{Z}}

\newcommand{\R}{\mathbb{R}}
\newcommand{\Q}{\mathbb{Q}}

\usepackage[english]{babel}

\newtheorem{theorem}{Theorem}

\newtheorem{lemma}[theorem]{Lemma}

\newtheorem{proposition}[theorem]{Proposition}
\newtheorem{corollary}[theorem]{Corollary}

\theoremstyle{definition}

\newtheorem{definition}[theorem]{Definition}
\newtheorem{assumption}[theorem]{Assumption}
\newtheorem{exmp}{Example}

\newtheorem{remark}[theorem]{Remark}

\newcommand{\cx}{[ \mathbf x]}

\DeclareMathOperator{\poly}{poly}
\newcommand{\eps}{\varepsilon}

\newcommand{\err}{\mathcal E}
\DeclareMathOperator{\rank}{rank}
\newcommand{\momt}{\mathrm{mom}(f)_t}
\newcommand{\sost}{\mathrm{sos}(f)_t}
\newcommand{\Sgh}{S(\mathbf{g}, \mathbf{h})}
\DeclareMathOperator{\vol}{vol}
\DeclareMathOperator{\dist}{dist}

\newcommand{\mg}{\mathbf{g}}
\newcommand{\mh}{\mathbf{h}}
\newcommand{\mx}{\mathbf{x}}

\copyrightyear{2023} 
\acmYear{2023} 
\setcopyright{rightsretained} 
\acmConference[ISSAC 2023]{International Symposium on Symbolic and Algebraic Computation 2023}{July 24--27, 2023}{Tromsø, Norway}
\acmBooktitle{International Symposium on Symbolic and Algebraic Computation 2023 (ISSAC 2023), July 24--27, 2023, Tromsø, Norway}\acmDOI{10.1145/3597066.3597075}
\acmISBN{979-8-4007-0039-2/23/07}

\author{Sander Gribling}
\email{gribling@irif.fr}
\affiliation{%
  \institution{IRIF, Université Paris Cit\'e}
  \city{Paris}
  \country{France}
}

\author{Sven Polak}
\email{s.c.polak@tilburguniversity.edu}
\affiliation{%
  \institution{Tilburg University}
  \city{Tilburg}
  \country{The Netherlands}
}

\author{Lucas Slot}
\email{lucas.slot@inf.ethz.ch}
\affiliation{%
  \institution{ETH Z\"urich}
  \city{Z\"urich}
  \country{Switzerland}
}

\title[A note on the computational complexity of the moment-SOS hierarchy]{A note on the computational complexity of the \\moment-SOS hierarchy for polynomial optimization}

\begin{abstract} 
The moment-sum-of-squares (moment-SOS) hierarchy is one of the most celebrated and widely applied methods for approximating the minimum of an $n$-variate polynomial over a feasible region defined by polynomial (in)equalities. A key feature of the hierarchy is that, at a fixed level, it can be formulated as a semidefinite program of size polynomial in the number of variables $n$. Although this suggests that it may therefore be computed in polynomial time, this is not necessarily the case. Indeed, as O'Donnell~\cite{ODonnell:bitcomplexity} and later Raghavendra \& Weitz~\cite{RaghavendraWeitz:bitcomplexity} show, there exist examples where the sos-representations used in the hierarchy have exponential bit-complexity. We study the computational complexity of the moment-SOS hierarchy, complementing and expanding upon earlier work of Raghavendra \& Weitz~\cite{RaghavendraWeitz:bitcomplexity}. In particular, we establish algebraic and geometric conditions under which polynomial-time computation is guaranteed to be possible. 
\end{abstract}


\begin{document}
\sloppy

\maketitle

\section{Introduction}
Consider the polynomial optimization problem:
\begin{equation}
\begin{split} \label{EQ:POP}
      f_{\min} := \min ~& f(\mx) \\
      \mathrm{s.t. }~&g_i(\mx) \geq 0 \quad(1 \leq i \leq m), \\
      &h_j(\mx) = 0 \quad(1 \leq j \leq \ell), \\
      &\mx \in \R^n,
\end{split}  \tag{POP} 
\end{equation}
where $f, g_i, h_j \in \R[\mx]$ are given $n$-variate polynomials.
The feasible region of \cref{EQ:POP} is a basic \emph{semialgebraic} set, which we denote by:
\[
    S(\mathbf{g}, \mathbf{h}) := \{ \mx \in \R^n : g_i(\mx) \geq 0,~
      h_j(\mx) = 0\}.
\]
Problems of the form \cref{EQ:POP} are generally hard and non-convex. They naturally capture several classical combinatorial problems, and have applications in finance, energy optimization, machine learning, optimal control and quantum computing. As they are often intractable, several techniques have been proposed to approximate them. Perhaps the most well-known and studied among these techniques is the so-called \emph{moment-SOS hierarchy}, due to Lasserre~\cite{Lasserre:seminal} and Parillo~\cite{Parillo:sos}. The main idea behind the hierarchy is that one %may verify 
can certify the nonnegativity of a polynomial $p \in \R[\mx]$ on $S(\mathbf{g}, \mathbf{h})$ by representing it as a \emph{weighted sum of squares}:
\begin{equation} \label{EQ:SOSRepresentation}
        p(\mx) = \sum_{i=0}^m g_i(\mx) \sigma_i(\mx) + \sum_{j=1}^\ell h_j(\mx) p_j(\mx),
\end{equation}
where $\sigma_i \in \Sigma[\mx]$ are sums of squares, $p_j \in \R[\mx]$ and we set ${g_0(\mx) = 1}$ for convenience. We say that a representation~\cref{EQ:SOSRepresentation} is of degree $t$ if $\mathrm{deg}(g_i\sigma_i) \leq t$ and $\mathrm{deg}(h_jp_j) \leq t$ for all $i,j$.
For $t \in \N$, one then obtains a lower bound $\mathrm{sos}(f)_t \leq f_{\min}$ on the minimum of $f$ by:
\begin{equation}\label{EQ:introductionSOS}  \tag{SOS}
\begin{split}
\mathrm{sos}(f)_t := \sup_{\lambda \in \R} \big\{\lambda :\, & 
    f - \lambda~\text{has a representation~\cref{EQ:SOSRepresentation}}
    \\[-5pt]&\text{ of degree } 2t \big\}.
\end{split}
\end{equation}
% We note that \cref{EQ:introductionSOS} is infeasible for $t<\deg(f)/2$.
% % , and by convention we say that it then has value $-\infty$.
% We therefore assume from now on that $t \geq \lceil \deg(f)/2\rceil$. 
For fixed level~$t$, this lower bound may be computed by solving a semidefinite program (SDP) involving matrices of size polynomial in $n$.
It is often claimed that one may therefore (approximately) compute $\mathrm{sos}(f)_t$ in polynomial time, for instance by applying the ellipsoid algorithm. As was noted by O'Donnell~\cite{ODonnell:bitcomplexity} and later by Raghavendra \& Weitz~\cite{RaghavendraWeitz:bitcomplexity}, this is not necessarily the case. Indeed, polynomial runtime of the ellipsoid algorithm is only guaranteed when the feasible region of the SDP contains an \emph{inner ball} which is not too small, and is contained in an \emph{outer ball} which is not too large. Informally, these two balls ensure %this means 
that it is possible to choose the coefficients of the multipliers $\sigma_i, p_j$ in the representation~\cref{EQ:SOSRepresentation} so that their \emph{bit-complexity} is polynomial in $n$. We call such a representation \emph{compact}. 
% As we explain in more detail below, 
Roughly speaking,
the following mild algebraic boundedness %compactness 
assumption on the feasible region $S(\mg, \mh)$ guarantees the existence of the inner ball for the SDP-formulation of $\sost$.
\begin{definition}
We say that the problem \cref{EQ:POP} is \emph{explicitly  bounded} 
if $g_1(\mx) = R^2 - \| \mx \|_2^2$ for some $R \geq 0$. 
\end{definition}
\begin{remark}
If \cref{EQ:POP} is explicitly bounded, then \cref{EQ:introductionSOS} is feasible for all $t \geq \lceil \deg(f)/2\rceil$. 
\end{remark}
\begin{remark}
Explicit boundedness is a slightly stronger assumption than the usual Archimedean condition, which merely requires ${R^2 - \| \mx \|_2^2}$ to have a representation of the form \cref{EQ:SOSRepresentation}. For example, the semialgebraic set defined by $x^2(1-x^2) \geq 0$ is Archimedean\footnote{{Technically, both the Archimedean and explicit boundedness condition are properties of the description of a semialgebraic set and not of the set itself. However since we always fix the description, we often write `the set $S(\bm g,\bm h)$ is explicitly bounded'.}} since we have the identity $1-x^2 = (1-x^2)^2 + x^2(1-x^2)$, but it clearly does not have a constraint of the form $R^2 - x^2$ in its description. Likewise, we point out that a compact semialgebraic set is not necessarily Archimedean (see, e.g.~\cite[Ex.~6.3.1]{Prestel:polopt},~\cite[Ex.~3.19]{Laurent:polopt}).  \end{remark}

The remaining question, then, is whether an outer ball always exists. O'Donnell~\cite{ODonnell:bitcomplexity} shows that in fact, it does not; he constructs an example where every representation \cref{EQ:SOSRepresentation} of $f(\mx) - \mathrm{sos}(f)_2$ necessarily involves multipliers $\sigma_i, h_j$ whose coefficients are doubly-exponentially large in~$n$. Raghavendra \& Weitz~\cite{RaghavendraWeitz:bitcomplexity} subsequently show that it is possible to construct such an example even when the equalities $\mh$ include the boolean constraints ${\mx_i - \mx_i^2 = 0}$, negatively answering a question posed by O'Donnell~\cite{ODonnell:bitcomplexity}. On the positive side, they show conditions under which existence of a compact representation \cref{EQ:SOSRepresentation} is guaranteed. These conditions are met for the reformulation of several well-known combinatorial problems as a \cref{EQ:POP}, as well as for optimization over the unit hypersphere. %In order for their result to make sense, we must 
To state our and their results, we make the natural assumption that the coefficients of the objective function $f$ and the polynomials $g_i, h_j$ defining the feasible region $S(\mg, \mh)$ of~\cref{EQ:POP} have polynomial bit-complexity.
\begin{assumption}
Throughout, we assume that the coefficients of the polynomials $f, g_i, h_j$ in \cref{EQ:POP} have polynomial bit-complexity in $n$ and that their degree is independent of~$n$. We also assume that the number of constraints ($m+\ell$) is polynomial in $n$.
\end{assumption}

\begin{theorem}[Main positive result of \cite{RaghavendraWeitz:bitcomplexity}, paraphrased] \label{THM:RW}
Let $S(\mathbf{g}, \mathbf{h})$ be a semialgebraic set and let $t \geq \lceil \deg(f)/2\rceil$ be fixed. Assume that $S(\mathbf{g}, \mathbf{h})$ is explicitly bounded: $g_1(\mx) = R^2 - \sum_{i=1}^n \mx_i^2$ for some $1 \leq R \leq 2^{\mathrm{poly}(n)}$. Suppose furthermore that the following conditions are satisfied:
\begin{enumerate}
% \item The set $S(\mathbf{g}, \mathbf{h})$ satisfies the explicitly boundedness condition: $g_1(\mx) = R^2 - \sum_{i=1}^n \mx_i^2$ for some $1 \leq R \leq 2^{\mathrm{poly}(n)}$.
\item \label{cond:1} For any $p \in \R[\mx]_{t}$ %2t}$ 
with $p(\mx) = 0$ for all $\mx \in S(\mathbf{g}, \mathbf{h})$, there are $p_1, p_2, \ldots, p_\ell\in\R[\mx]$ such that:
\[
	p(\mx) = \sum_{j=1}^\ell p_j(\mx) h_j(\mx),
\]
and $\mathrm{deg}(p_jh_j) = O(t)$.
\item \label{cond:2} Let $\mu$ be the uniform probability measure on $S(\mathbf{g}, \mathbf{h})$. The \emph{moment matrix} $M_t(\mu)$ defined by:
\[
	M_t(\mu)_{\alpha, \beta} := \int_{S(\mathbf{g}, \mathbf{h})} \mx^{\alpha+\beta} d\mu(\mx) \quad(\alpha, \beta \in \N^n_{t})
\]
has smallest non-zero eigenvalue $\geq 2^{-\mathrm{poly}(n)}$.
\item \label{cond:3} There exists an $\eta \geq 2^{-\mathrm{poly}(n)}$ such that $g_i(\mx) \geq \eta$ for all $\mx \in S(\mathbf{g}, \mathbf{h})$ and $1 \leq i \leq m$.
\end{enumerate}
Then the program~\cref{EQ:introductionSOS} has an (approximately) optimal solution involving only multipliers $\sigma_i, p_j$ whose coefficients are at most $2^{\mathrm{poly}(n)}$.
\end{theorem}

The conditions of \cref{THM:RW} have a natural interpretation in the dual formulation of~\cref{EQ:introductionSOS}, the \emph{moment formulation}, which  reads (see, e.g., \cite{deKlerkLaurent:survey}):
\begin{equation} \label{EQ:introductionMOM} \tag{MOM}
\begin{split}
\mathrm{mom}(f)_t := \inf_{L \in \R[\mx]_{2t}^*} \big\{L(f):\,& L(1)=1, L(p) \geq 0 \text{ for } p 
\notag \\[-5pt]& \text{with a representation~\cref{EQ:SOSRepresentation}}\notag\\&\text{of degree } 2t \big\}.
\end{split}
\end{equation}
We refer to \cref{sec:L} for a detailed discussion of \cref{EQ:introductionMOM} and the duality between it and \cref{EQ:introductionSOS}. For now let us introduce two types of matrices associated to linear functionals $L \in \R[\mx]_{2t}^*$: the \emph{moment matrix} $M_t(L)$ and, for a polynomial $g$, the localizing matrix $M_t(gL)$. Their entries are given by 
\begin{alignat}{2}
    \label{EQ:LM2}
    M_t(L)_{\alpha, \beta}       &= L(\mx^{\alpha + \beta}) \quad&&(|\alpha|,|\beta| \leq t), \\ %\mathrm{deg}(\mx^{\alpha+\beta}) \leq 2t), \\
    \label{EQ:LMg2}
    M_t(gL)_{\alpha, \beta}  &= L(g(\mx) \mx^{\alpha+\beta}) \quad&&(|\alpha|,|\beta| \leq t -\lceil\deg(g)/2\rceil). %\mathrm{deg}(g_i(\mx) \mx^{\alpha+\beta}) \leq 2t).
\end{alignat}
For ease of notation, we use a subscript $t$ for the localizing matrix $M_t(gL)$ even though it is indexed by monomials of degree at most $t -\lceil\deg(g)/2\rceil$.
Here we use terminology associated to measures since the linear functionals $L$ in \cref{EQ:introductionMOM} can be viewed as relaxations of probability measures: any probability measure $\mu$ whose support is contained in $\Sgh$ gives rise to a feasible $L_{\mu}$ defined by 
\begin{equation}\label{EQ:measureoperator}
L_{\mu}(p) = \int_{\Sgh} p(\mx) d\mu(\mx) \text{ for } p \in \R[\mx]_{2t}.
\end{equation}
We can then interpret the conditions in \cref{THM:RW} as follows: \eqref{cond:2}~says that $M_t(L_\mu)$ has smallest non-zero eigenvalue $\geq 2^{-\poly(n)}$ and~\eqref{cond:3} implies that the same holds for all localizing matrices $M_t(gL_\mu)$ with $g \in \mg$.

\subsection{Our contributions}
The main goal of this paper is to carefully map under what circumstances computation of the bounds $\mathrm{sos}(f)_t$ (and/or $\momt$) and the corresponding representation~\cref{EQ:introductionSOS} is guaranteed to be possible in polynomial time. 

Our starting point will be the following proposition which concerns the bounds $\momt$. The proof of this proposition consists of a straightforward reformulation of the conditions of strict feasibility for the SDP formulation of $\momt$ for explicitly bounded polynomial optimization problems, see \cref{sec:SDPformulation}. As we will see there, the conditions (1) and (2) in \cref{PROP:main} below will guarantee the existence of a ball in the feasible region of \eqref{EQ:LMOM}, whose radius depends on the smallest non-zero eigenvalue of the (localizing) moment matrices corresponding to a feasible solution $L$.
\begin{proposition} \label{PROP:main}
Let $S(\mathbf{g}, \mathbf{h})$ be a semialgebraic set and let $t \geq \lceil \deg(f)/2\rceil$ be fixed. Assume that $S(\mathbf{g}, \mathbf{h})$ is explicitly bounded: $g_1(\mx) = R^2 - \sum_{i=1}^n \mx_i^2$ for some $1 \leq R \leq 2^{\mathrm{poly}(n)}$. Suppose furthermore that there exists an $L \in \R[\mx]_{2t}^*$ with $L(1)=1$ and the following properties:
\begin{enumerate}
    \item For any $g \in \mathbf g$ and any  $p \in \R\cx_{t-\lceil \deg(g)/2\rceil}$, if $L(g p^2) =0$, then there are $p_1, p_2, \ldots, p_\ell\in\R[\mx]$ such that:
\[
	g p^2(\mx) = \sum_{j=1}^\ell p_j(\mx) h_j(\mx),
\]
and $\deg(p_j) \leq 2t-\deg(h_j)$ for each $j \in [\ell]$. We recall that $1 \in \mg$ by convention.
\item The matrices $M_t(L)$ and $M_t(gL)$ ($g \in \mathbf g$) have smallest non-zero eigenvalue at least $2^{-\poly(n)}$. 
\end{enumerate}
Then for $\varepsilon \geq 2^{-\mathrm{poly}(n)}$, the bound $\mathrm{mom}(f)_t$ (which equals $\mathrm{sos}(f)_t$) may be computed in polynomial time in $n$ up to an additive error of at most $\varepsilon$.
\end{proposition}

The statement of \cref{PROP:main} is stronger than the result of Raghavendra \& Weitz in the sense that it guarantees polynomial-time computation of the bound $\mathrm{sos}(f)_t$, whereas Theorem~\ref{THM:RW} only guarantees existence of a compact representation \cref{EQ:SOSRepresentation}.
Furthermore, its conditions do not require \emph{strict positivity} of the inequality constraints $\mg$ on $S(\mg, \mh)$. As we see below, it therefore applies to several natural settings where \cref{THM:RW} may not be applied. On the other hand, the first condition of \cref{PROP:main} is more restrictive than the first condition of \cref{THM:RW}. We note that it is satisfied for example when $L$ is the linear operator associated to a positive Borel measure supported on $\Sgh$, and the constraints $\mh$ form a Gr\"obner basis of a real radical ideal (cf.~\cite[Sec.~2]{Laurent:polopt}). %\SG{Proof sketch: Positive Borel measure on the \SG{real} variety implies ``$L(p)=0$ implies $p=0$''. The real Nullstellensatz shows $\mathcal I(V_\R(\mathcal I)) = \sqrt[\R]{\mathcal I}$. An ideal $\mathcal I$ is real radical if $\sqrt[\R]{\mathcal I} = \mathcal I$. (So far, see e.g.~\cite[Sec.~2.1]{Laurent:polopt}.) For a Gr\"obner basis of $\mathcal I$, the division algorithm `just works'; see Thrm.~3 in Sec.~2.3 and Cor.~2 in Sec.~2.6 of \cite{Cox:ideals}. Alternatively, for the required properties of Gr\"obner bases see \cite[Sec.~2.3]{Laurent:polopt}. Interestingly, \cite[Thrm.~2.4]{Laurent:polopt} is similar in spirit to what we need here, but it only requires the weaker condition `radical ideal'.} 

Our first contribution is a sufficient condition for the second requirement of \cref{PROP:main}.

\begin{theorem} \label{THM:SUFFICIENT}
    Let $\Sgh$ be an explicitly bounded semialgebraic set with $R \leq 2^{\poly(n)}$ and let $L \in \R[\mx]_{2t}^*$ be a feasible solution to $\momt$ for $t \in \N$ fixed. Assume that $L(\mx^{\alpha}) \in \Q$ has polynomial bit-complexity for all $\alpha \in \N^n_{2t}$. Then the smallest non-zero eigenvalue of $M_t(L)$ is at least $2^{-\poly(n)}$ and the same holds for the localizing matrices $M_t(gL)$ for $g \in \mathbf g$.
\end{theorem}

Our second contribution is an alternative, \emph{geometric} condition on the feasible region $S(\mathbf{g}, \mathbf{h})$ of \cref{EQ:POP} which guarantees polynomial-time computation of $\mathrm{sos}(f)_t$ in the special case where the formulation does not contain any equality constraints. We write $S(\mathbf g)$ for $S(\mathbf g, \emptyset)$.
\begin{theorem}\label{THM:FULLDIM}
Let $S(\mathbf{g}) \subseteq \R^n$ be a semialgebraic set defined only by inequalities. Assume that the following two conditions are satisfied:
\begin{enumerate}
    \item $S(\mathbf{g})$ is explicitly bounded: $g_1(\mx) = R^2 - \sum_{i=1}^n \mx_i^2$ with constant $1 \leq R \leq 2^{\mathrm{poly}(n)}$.
    \item $S(\mathbf{g})$ contains a ball of radius $r \geq 2^{-\mathrm{poly}(n)}$, i.e., $B(z, r) \subseteq S(\mathbf{g})$ for some $z \in \R^n$.
\end{enumerate}
Then, for fixed $t \geq \lceil \deg(f)/2\rceil$ and $\varepsilon \geq 2^{-\mathrm{poly}(n)}$, the bound $\mathrm{sos}(f)_t$ may be computed in polynomial time in~$n$ up to an additive error of at most $\varepsilon$. 
\end{theorem}
% Note that we could demand equivalently in \cref{THM:FULLDIM} that $S(\mg)$ contains a ball of radius $r \geq 2^{-\poly(n)}$ centered at $z \in \R^n$.
The inclusions $B(z,r) \subseteq S(\mathbf g) \subseteq B(z,R)$ for $\frac{1}{r}, R \leq 2^{\poly(n)}$ are a natural way to ensure that $f$ has an approximate minimizer over $S(\mathbf g)$ whose bit-complexity is $\poly(n)$. Furthermore, they are very reminiscent of the sufficient conditions for solving semidefinite programs in polynomial time, see \cref{THM: deKlerkVallentin} below. 

As we will see in Proposition~\ref{PROP:smallcube}, it is possible to choose constraints $\mathbf g = (g_i)$, each with \emph{constant} bit-complexity, such that the second condition of Theorem~\ref{THM:FULLDIM} is not satisfied. Notably, the resulting semialgebraic set $S(\mathbf{g})$ does not satisfy the conditions of Theorem~\ref{THM:RW} or Theorem~\ref{PROP:main} either.

Finally, as a third contribution, we make explicit the connection between computational aspects of the primal formulation \cref{EQ:introductionSOS} of the sos-hierarchy, and its dual formulation \cref{EQ:LMOM} in terms of %\emph
{moments}. % (see below). 
This connection is implicitly present in the proof of Theorem~\ref{THM:RW} in~\cite{RaghavendraWeitz:bitcomplexity}. 
% It is of importance because in practice, one often uses the dual formulation to compute the bound $\mathrm{sos}(f)_t$.
\begin{theorem} \label{THM:CONNECTION}
Let $S(\mathbf{g}, \mathbf{h})$ be a semialgebraic set and suppose that the conditions of Theorem~\ref{PROP:main} or Theorem~\ref{THM:FULLDIM} are satisfied. 
Then, for a fixed $t \geq \lceil \deg(f)/2\rceil$ and $\eps>0$, there exists a sum-of-squares representation~\cref{EQ:SOSRepresentation} proving nonnegativity of $f - \mathrm{sos}(f)_t+\eps$ on $S(\mathbf{g}, \mathbf{h})$ with bit-complexity $\poly(n,\log(1/\eps))$. 
\end{theorem}

\section{On the geometric condition} \label{SEC:Applications}
Before we move on to the proofs of our results, let us give some examples of natural settings where they may be applied. In general, Theorem~\ref{THM:FULLDIM} is better equipped to deal with non-discrete semialgebraic sets $S(\mg, \mh)$ than Theorem~\ref{THM:RW}. The third condition of Theorem~\ref{THM:RW}, which demands in particular that $g_i(\mx) > 0$ for each $\mx \in S(\mg, \mh)$, is rather hard to satisfy: 
\begin{exmp}
The unit hypercube $[-1, 1]^n$, the unit ball $B^n \subseteq \R^n$ and the standard simplex $\Delta^n \subseteq \R^n$ are semialgebraic sets, defined by:
\begin{align*}
    [-1, 1]^n &= \{ \mx \in \R^n : 1 - \mx_i^2 \geq 0, \quad i=1, 2, \ldots, n\}, \\
    B^n &= \{ \mx \in \R^n : 1 - \|\mx\|_2^2 \geq 0 \}, \\
    \Delta^n &= \{ \mx \in \R^n : \mx_i \geq 0, ~1 - \textstyle\sum_{i=1}^n \mx_i \geq 0 \}.
\end{align*}
It is straightforward to see that they each satisfy the conditions of Theorem~\ref{THM:FULLDIM} (after adding a ball constraint $g_1(\mx) = R^2 - \|\mx\|_2^2$ if needed). They do not, however, satisfy the third condition of Theorem~\ref{THM:RW}.
\end{exmp}

The following proposition gives an alternative sufficient condition related to strict feasibility of~\cref{EQ:POP}, which implies that the conditions for Theorem~\ref{THM:FULLDIM} are satisfied and thus guarantees polynomial-time computability of the bound $\momt$. 
\begin{proposition} \label{prop:strict feas low bit-complexity}
Let $S(\mg)$ be a full-dimensional semialgebraic set contained in a ball of radius $R \leq 2^{\poly(n)}$. Assume that there exists a rational point $\mx \in S(\mg)$ with $g_i(\mx) > 0$ for all $i$, whose bit-complexity is polynomial in $n$. Then $S(\mg)$ contains a ball of radius $r \geq 2^{-\poly(n)}$. 
\end{proposition}
\begin{proof}
As the polynomials $g_i$ are of fixed degree and have bounded coefficients, their Lipschitz constants on $B(0, R)$ can be bounded by $2^{\poly(n)}$. Furthermore, as $\mx$ has polynomial bit-complexity in $n$, we know that ${g_i(\mx) \in \Q}$ has polynomial bit-complexity as well, and therefore $g_i(\mx) > 0$ implies that $g_i(\mx) \geq 2^{-\poly(n)}$. Together, this implies that there exists an $r \geq 2^{-\poly(n)}$ such that $g_i(\mathbf y) \geq 0$ for all $\mathbf y$ with ${\| \mx - \mathbf y\|_2 \leq r}$, i.e., such that $S(\mg)$ contains the ball of radius $r \geq 2^{-\poly(n)}$ centered at~$\mx$.
\end{proof}

\subsection{Semialgebraic sets with large volume contain a large ball}
Another class of semialgebraic sets satisfying the conditions of~\cref{THM:FULLDIM} are those of sufficiently large volume. In the case that $S(\mg)$ is convex, we have the following consequence of John's theorem.
\begin{lemma} \label{LEM:John}
Let $S(\mathbf g) \subseteq \R^n$ be a \emph{convex} semialgebraic set. Assume that $S(\mg)$ is contained in a ball of radius $R \leq 2^{\poly(n)}$. % and that $\vol(S(\mathbf g)) \geq 2^{-\poly(n)}$. 
Then there exists a constant
\[r \ge \sqrt{\vol(S(\mg))}  \cdot 2^{-\poly(n)}
\]
such that $S(\mg)$ contains a (translated) ball of radius $r$. 
\end{lemma}
\begin{proof}
    By John's theorem~\cite{John:ellipsoid} there exists an ellipsoid $E$ with center $c \in \R^n$ such that 
    \[
    E \subseteq S(\mathbf g) \subseteq c+ n(E-c).
    \]
    These inclusions show that $n^n\vol(E) = \vol(c+n(E-c)) \geq \vol(S(\mathbf g))$ and thus $\vol(E) \geq \vol(S(\mathbf g)) n^{-n}$. Let $E = \{x \in \R^n: (x-c)^T A^{-1} (x-c) \leq 1\}$ for a positive definite matrix $A$. Let $v$ be an eigenvector of $A$ corresponding to eigenvalue $\lambda$. Then $c \pm \sqrt{\lambda} v \in E \subseteq B(0,R)$ and hence $\lambda \leq R^2$. Moreover, $\vol(E) = \vol(B(0,1)) \prod_{i=1}^n \sqrt{\lambda_i}$ where $\lambda_1,\ldots,\lambda_n$ are the eigenvalues of $A$. Combined with the lower bound on $\vol(E)$, this shows for each $i \in [n]$ that 
    \[
    \lambda_i \geq \vol(S(\mathbf g)) n^{-n} (R^2)^{-(n-1)} 
    \vol(B(0,1))^{-1}.
    \]
    We thus have $B(z,r) \subseteq E \subseteq S(\mathbf g)$ for $z = c$ and
    \[
        r = \left(\frac{\vol(S(\mathbf g))}{n^{n} R^{2(n-1)} \vol(B(0,1))}\right)^{\frac{1}{2}} \geq \sqrt{\vol(S(\mathbf g))} \cdot 2^{-\poly(n)}. \qedhere    \]
\end{proof}
\Cref{LEM:John} can be used to show that full-dimensional, rational polytopes contain a large enough ball % the conditions of \Cref{THM:FULLDIM} 
(which is well known, see, e.g., \cite[Thm.~3.6]{schrijveretal:ellipsoidpaper}, \cite[Lem.~1-2]{GacsLovasz:polytope}).
\begin{corollary}
Let $P \subseteq \R^n$ be a full-dimensional polytope defined by rational linear inequalities with polynomial bit-complexity. Then $P$ is contained in a ball of radius $R \leq 2^{\poly(n)}$ and ${\vol P \geq 2^{-\poly(n)}}$. Using \cref{LEM:John}, $P$ thus contains a ball of radius $r \geq 2^{- \poly(n)}$. 
\end{corollary}
\begin{proof}
Let $\{p_1, \ldots, p_N\}$ be the vertices of $P$, each of which has polynomial bit-complexity in $n$, as a consequence of Cramer's rule. As $P$ is full-dimensional, we may assume w.l.o.g. that $0 \in P$ and that $p_1, p_2, \ldots p_n$ are linearly independent. Therefore, 
\[
    0 < \vol \big(\mathrm{span} \{ p_1, \ldots p_{n}\} \big) = \frac{1}{n!} |\det \big (p_1~p_2~\ldots~p_{n} \big)|.
\]
But since the $p_i$'s have polynomial bit-complexity in $n$, we may conclude that 
\[
\vol P %\geq 2^{-\poly(n)} 
\geq \frac{1}{n!} |\det \big (p_1~p_2~\ldots~p_{n} \big)| \geq 2^{-\poly(n)}. \qedhere
\]
\end{proof}

In fact, a result similar to \cref{LEM:John} can be shown \emph{without} the assumption that $S(\mg)$ is convex. That is, \emph{any} bounded semialgebraic set $S(\mg) \in B(0,R)$ with $R \leq 2^{\poly(n)}$ and ${\vol (S(\mg)) \geq 2^{-\poly(n)}}$ satisfies the conditions of \cref{THM:FULLDIM}. 
% Semialgebraic sets that contain a (translated) cube $[-r,r]^n$ for some $r \geq 2^{-\poly(n)}$ naturally have volume at least $2^{-\poly(n)}$. Here we show the converse. 
To show this, we use an upper bound on the volume of neighborhoods of algebraic varieties from \cite{Basu:hausdorff}. Obtaining such volume bounds is a problem with a long history, see, e.g.,~\cite{Wongkew:tubular,Lotz:tubular,Basu:hausdorff} for a discussion. 
\begin{proposition} \label{PROP:largevolume}
Let $S(\mathbf g) %\subseteq B(0,R) 
\subseteq \R^n$.  %be an explicitly bounded set 
Assume that $S(\mg)$ is contained in a ball of radius $R \leq 2^{\poly(n)}$. %with $0<R<2^{\poly(n)}$. 
Then there exists a constant \[r \geq \vol(S(\mg)) \cdot 2^{-\poly(n)}\] such that $S(\mg)$ contains a (translated) ball of radius~$r$. 
\end{proposition}
\begin{proof} %[Proof sketch]
Let $G = \prod_{g \in \mg} g$. First observe that the boundary of $S(\mg)$ is contained in the variety $V(G) = \{\mx \in \R^n: G(\mx) = 0\}$. %Indeed, for every point $x$ on the boundary of $S(\mg)$ at least one of the constraints $g(x) \geq 0$ holds with equality, and hence $G(x)=0$. 
Now consider the $\delta$-neighborhood of $V(G)$, i.e., the set of $\mx \in \R^n$ such that $\dist(\mx,V(G)) \leq \delta$. In \cite[Theorem~3.2]{Basu:hausdorff}, building on \cite{Lotz:tubular}, it is shown that for all positive $\delta$ we have
\[
\frac{\vol(V(G)+B(0,\delta))}{\vol(B(0,R))} \leq 4 \sum_{i=1}^{n} \binom{n}{i} \left(\frac{4D\delta}{R}\right)^{i} \left(1+\frac{\delta}{R}\right)^{n-i}
\]
% \[
% \vol(V(G)+B(0,\delta)) \leq \sum_{j=1}^{n} c_j d^j \delta^j R^{n-j}
% \]
where $D = \prod_{g \in \mg} \deg(g)$.\footnote{Theorem~3.2 of \cite{Basu:hausdorff} applies to more general algebraic varieties $V$, here we use the special case where $V$ is the zero set of a single polynomial and $\mathrm{dim}_\R(V) = n-1$.} 
%From the proof of this bound provided in \cite{Wongkew:tubular} one can extract an upper bound on $c_j$ of the form $c_j \leq 2^{\poly(n)}$.\footnote{The proof provided in \cite{Wongkew:tubular} is by induction on the number of variables $n$. The case $n=1$ is easy to show. For the induction step, let $\{c^{n-1}_j\}_{1 \leq j \leq n-1}$ be the coefficients for varieties defined by polynomials in $n-1$ variables. Then~\cite[Eq.~(7)]{Wongkew:tubular} shows that for varieties defined by polynomials in $n$ variables we can take $c^{n}_j = 2n^2 (1+\frac{1}{n})^j c^{n-1}_j$ for $1 \leq j \leq n-1$. Moreover, for the final coefficient, \cite[Eq.~(9)]{Wongkew:tubular} then shows that one can take $c^n_n = 3^n \vol(B(0,1))$.} 
In particular, this shows that 
\[
\vol\left(V(G)+B(0,\delta)\right) \leq \sum_{j=1}^n  c_j \delta^j
\]
where the coefficients $c_j \in \R$ satisfy $0< c_j \leq 2^{\poly(n)}$. For $\delta \leq 1$, the upper bound is of the form $\delta \cdot 2^{p(n)}$ for some polynomial $p$, and hence for $\delta < \vol(S(\mg))/2^{p(n)}$ 
we have 
\[
\vol(\partial S(\mg) + B(0,\delta)) \leq \vol(V(G)+B(0,\delta)) < \vol(S(\mg)),
\]
which implies that there exists an $\mx \in S(\mg)$ that has distance greater than $\delta$ to the boundary of $S(\mg)$. The ball $B(\mx, r)$ with center $\mx$ and radius $r = \delta \geq \vol(S(\mg))/2^{\poly(n)}$ is thus contained in $S(\mg)$. 
\end{proof}

\subsection{A counterexample}
To end this section, we show that the second condition of Theorem~\ref{THM:FULLDIM} is not superfluous by exhibiting a full-dimensional semialgebraic set which does not contain a hypercube $[-r, r]^n$ of size $r \geq 2^{-\poly(n)}$. As a direct result, it cannot contain a ball of radius $r \geq 2^{-\poly(n)}$, either. We use a simple repeated-squaring argument. 
\begin{proposition} \label{PROP:smallcube}
There exists a full-dimensional semialgebraic set $S(\mathbf{g})$, defined only by polynomial inequalities $\mathbf{g} = (g_1, g_2, \dots, g_m)$ of fixed degree, whose coefficients have constant bit-complexity, % polynomial in~$n$, 
which does not contain a (translated) cube $[-r, r]^n$ for $r \geq 2^{-\mathrm{poly}(n)}$.
\end{proposition}
\begin{proof}
Let $S(\mathbf{g})$ be the set defined by the system of inequalities:
\begin{align*}
    \mx_i &\geq 0 \quad\quad(1 \leq i \leq n), \\
    \mx_i - \mx_{i+1}^2 &\leq 0 \quad\quad(1 \leq i \leq n-1),\\
    \mx_n &\leq 1/2.
\end{align*}
Set~$r := 2^{-2^{{n-1}}}$. Then~$[0,r]^n \subseteq S(\mathbf{g})$ and so $S(\mathbf{g})$ is full-dimensional. But from the inequalities it follows that~$0 \leq \mx_1 \leq r$ for any $\mx \in S(\mathbf{g})$, meaning $S(\mathbf{g})$ cannot contain a cube of size at least $2^{-\mathrm{poly}(n)}$. 
\end{proof}

\section{Preliminaries} \label{SEC:Prelim}
\subsection{Notation}
We denote by $\R\cx$ the space of $n$-variate polynomials. For $d \in \N$, we write $\R\cx_d$ for the subspace of polynomials of degree at most~$d$, whose dimension is equal to $h(n,d) := \binom{n + d}{d}$. It has a basis of monomials $\mx^\alpha = \mx_1^{\alpha_1} \ldots \mx_n^{\alpha_n}$, with $\alpha \in \N_d^n := {\{\alpha \in \N^n : \|\alpha\|_1 \leq d\}}$. We denote by $\mathcal S^n$ the space of $n$-by-$n$ real symmetric matrices. We have the trace inner product $\langle X, Y \rangle = \mathrm{Tr}(XY)$ on $\mathcal S^n$, which induces the Frobenius norm $\|X\|_F = \sqrt{\langle X, X \rangle}$.

\subsection{Complexity of semidefinite programming}
We will make use of the following result on the complexity of semidefinite programming.
Let $C \in \Q^{n \times n}$, $A_i \in \Q^{n \times n}$, $b_i \in \Q$, $1 \leq i \leq m$ be rational input data. Consider the following SDP in standard form:
\begin{equation}
\begin{split} \label{EQ:SDP}
      \mathrm{val} = \inf ~&\langle C, X \rangle \\
      &\langle A_i, X \rangle = b_i \quad(1 \leq i \leq m), \\
      &X \in \mathcal{S}^n \text{ is positive semidefinite.}
\end{split}\tag{SDP} 
\end{equation}
We denote the feasible region of \cref{EQ:SDP} by:
\[
    \mathcal{F} := \{ X \in \mathcal{S}^n : X \succeq 0, \langle A_i, X \rangle = b_i \quad (1 \leq i \leq m) \}.
\]
One can show, for example using the ellipsoid method~\cite{schrijveretal:ellipsoid}, that under certain assumptions \cref{EQ:SDP} can be solved in polynomial time. We use the following explicit formulation from \cite{deKlerkVallentin:SDPcomplexity} where a similar result was shown using an interior point method.

\begin{theorem}[Thm.~1.1 in \cite{deKlerkVallentin:SDPcomplexity}]\label{THM: deKlerkVallentin}
Let $r, R > 0$ be given and suppose that there exists an $X_0 \in \mathcal{F}$ so that:
\[
    B(X_0, r) \subseteq \mathcal{F} \subseteq B(X_0, R),
\]
where $B(X_0, \lambda)$ is the ball of radius $\lambda \in \R$ (in the norm $\| \cdot \|_F$) centered at $X_0$ in the subspace:
\[
    V = V(\mathcal{F}) := \{ X \in \mathcal{S}^n : \langle A_i, X \rangle = b_i \quad (1 \leq i \leq m) \}. 
\]
Then for any rational $\varepsilon > 0$ one can find a rational matrix $X^* \in \mathcal{F}$ such that:
\[
    \langle C, X^* \rangle - \mathrm{val} \leq \varepsilon,
\]
in time polynomial in $n, m, \log (R / r), \log(1/\varepsilon)$ and the bit-complexity of the input data $C, A_i, b_i$ and the feasible point $X_0$.
\end{theorem}

\subsection{Dual formulation and moments of measures} \label{sec:L}
It will be convenient to work with the dual formulation of~\cref{EQ:introductionSOS}, which we recall is (see, e.g., \cite{deKlerkLaurent:survey}): 
\begin{equation}
\begin{alignedat}{2} \label{EQ:LMOM}
      \mathrm{mom}(f)_t := \inf ~& L(f)  \\
      \mathrm{s.t. }~&L(1) = 1, \\
      ~&L(g_i p^2) \geq 0,  ~~&&(g_i p^2 \in \R[\mx]_{2t}) \\
      &L(h_j \mx^\alpha) = 0, ~~   &&(h_j\mx^\alpha \in \R[\mx]_{2t}) \\
      &L \in \R[\mx]^*_{2t}.
\end{alignedat}
\tag{MOM}
\end{equation}
Assuming that \cref{EQ:POP} is explicitly bounded, these formulations are equivalent.
\begin{theorem}[{\cite[Sec.~2.6.3, Ex.~2.12]{BenTalNemirovski}, see also \cite{JoszHenrion:duality}}]
If \cref{EQ:POP} is explicitly bounded\footnote{If we assume instead that $S(\mg,\mh)$ is Archimedean with non-empty interior, then it is known that $\sost = \momt$ for $t$ large enough~\cite[Theorem~4.2]{Lasserre:seminal}.}, we have strong duality between the primal and dual formulations \cref{EQ:introductionSOS} and \cref{EQ:LMOM} of the moment-SOS hierarchy. That is, we then have:
\[
    \mathrm{sos}(f)_t = \mathrm{mom}(f)_t \quad \forall t \geq \lceil \deg(f)/2 \rceil.
    % ,\deg(g_i)/2,\deg(h_j)/2\}
\]
\end{theorem}

As is well known (see e.g., \cite[Prop.~6.2]{Laurent:polopt}), explicit boundedness of $S(\mg , \mh)$ also gives a bound on the feasible region of~\cref{EQ:LMOM}. 
\begin{lemma} \label{LEM:Lupbound}
Assume that $S(\mg, \mh)$ is explicitly bounded for some $R>0$. Let $L \in \R[x]^*_{2t}$ be a feasible solution to~\cref{EQ:LMOM}. Then ${|L(\mx^\alpha)| \leq R^{|\alpha|}}$ for all $\alpha \in \N^n_{2t}$. 
\end{lemma}
There is a natural relation between the dual formulation~\cref{EQ:LMOM} and \emph{moments} of measures supported on $S(\mg, \mh)$, which clarifies the assumptions made in Theorem~\ref{THM:RW} and Theorem~\ref{PROP:main}. For a measure $\mu$ supported on $S(\mathbf{g}, \mathbf{h})$, the \emph{moment} of degree $\alpha \in \N^n$ is defined by:
\[
	m(\mu)_{\alpha} := \int_{S(\mathbf{g}, \mathbf{h})} \mx^\alpha d\mu(\mx).
\]
For $t \in \N$, the (truncated) moment matrix $M_t(\mu)$ of order $t$ for $\mu$ is the matrix of size $h(n,t) = \binom{n + t}{t}$ given by:
\begin{equation} \label{Trunc}
	M_t(\mu)_{\alpha, \beta} = m(\mu)_{\alpha + \beta} \quad (|\alpha|,|\beta| \leq t).%= \int_{S(\mathbf{g}, \mathbf{h})} \mx^\alpha \mx^\beta d\mu(\mx) \quad (\alpha, \beta \in \N^n_t).
\end{equation}
% For a Borel probability measure $\mu$ supported on $S(\mathbf{g}, \mathbf{h}$), 
Consider the linear functional $L_\mu \in \R[\mx]_{2t}^*$ defined by:
\begin{equation} \label{EQ:measurefunctional}
    L_{\mu}(p) := \int_{S(\mathbf{g}, \mathbf{h})} p(\mx) d\mu(x) \quad(p \in \R[\mx]_{2t}).
\end{equation}
For any constraint $g_i$ and $p \in \R[\mx]$ with $\mathrm{deg}(g_i p^2) \leq 2t$, we have:
\[
     L_{\mu}(g_ip^2) = \mathbf{p}^\top M_t(g_i\mu) \mathbf{p} = \int_{S(\mathbf{g}, \mathbf{h})} p^2(\mx) g_i(\mx) d\mu(\mx) \geq 0,
\]
where $\mathbf{p}$ denotes the vector of coefficients of $p \in \R[\mx]_{t}$ in the monomial basis. Here the $(\alpha,\beta)$-entry of the \emph{localizing matrix} $M_t(g_i \mu)$ is defined as $\int_{S(\mathbf{g},\mathbf{h})} g_i(\mx) \mx^{\alpha+\beta} d\mu(\mx)$. In particular, for each $i$ the matrix $M_t(g_i \mu)$ is positive semidefinite. 
Note that $M_t(\mu) = M_t(L_\mu)$ and similarly for the localizing matrices.
Furthermore, for any constraint $h_j$ and $\alpha \in \N^n$ with $\mathrm{deg}(\mx^\alpha h_j) \leq 2t$, we have:
\[
     L_{\mu}(h_j\mx^\alpha) = \int_{S(\mathbf{g}, \mathbf{h})} h_j(\mx) \mx^\alpha d\mu(\mx) = 0.
\]
If $\mu$ is a probability measure, we get $L_{\mu}(1) = 1$, and it follows that $L_{\mu}$ is a feasible solution to \cref{EQ:LMOM}.

\subsection{Standard form of the moment formulation: Proof of Proposition~\ref{PROP:main}}
\label{sec:SDPformulation}

In order to apply~\cref{THM: deKlerkVallentin} in our proof of~\cref{PROP:main} below, we need to express the formulation~\cref{EQ:LMOM} in the standard form~\cref{EQ:SDP}. Set $N = h(n,t) + \sum_{i \in [m]} h(n,t-\lceil \deg(g_i)/2\rceil)$. One can construct $A_1,\ldots,A_K \in \Q^{N \times N}$ and $b_1,\ldots, b_K \in \Q$ with entries either $-1$, $0$, $1$, or a coefficient of the polynomials $g_i,h_j$, such that %every feasible solution $X$ is %
the conditions $\langle A_j, X \rangle = b_j$ for $j \in [K]$ are equivalent to the statement that $X$ is 
a block-diagonal matrix $X = M_t(L) \oplus \big(\bigoplus_{i \in [m]} M_t(g_iL)\big)$ where %d_iLet $L \in \R\cx_{2t}^*$, and consider the matrices $M, M(g_i)$ given by:
\begin{alignat}{2}
    \label{EQ:LM}
    M_t(L)_{\alpha, \beta}       &= L(\mx^{\alpha + \beta}) \quad&&(|\alpha|,|\beta| \leq t), \\ %\mathrm{deg}(\mx^{\alpha+\beta}) \leq 2t), \\
    \label{EQ:LMg}
    M_t(g_iL)_{\alpha, \beta}  &= L(g_i(\mx) \mx^{\alpha+\beta}) \quad&&(|\alpha|,|\beta| \leq t -\lceil\deg(g_i)/2\rceil), %\mathrm{deg}(g_i(\mx) \mx^{\alpha+\beta}) \leq 2t).
\end{alignat}
for some linear functional $L \in \R\cx_{2t}^*$ satisfying $L(1) = 1$ and $L(h_j\mx^\alpha)=0$ for each $j \in [\ell]$ and $|\alpha| \leq 2t-\deg(h_j)$. Let $C_f \in \mathcal S^N$ be such that $\langle C_f, X\rangle = L(f)$, then we have 
\begin{equation}
\begin{split} \label{eq:MOMSDP}
      \mathrm{mom}(f)_t = \inf ~&\langle C_f, X \rangle  \\
      &\langle A_i, X \rangle = b_i \quad(1 \leq i \leq K), \\
      &X \in \mathcal{S}^N \text{ is positive semidefinite.}
\end{split}\tag{MOM-SDP} 
\end{equation}
To see the equivalence it suffices to observe that the conditions $L(g_i p^2) \geq 0$ for all $g_i$ and $p$ of appropriate degree of \cref{EQ:LMOM} are equivalent to $X \succeq 0$. 

The upshot is that we may apply \cref{THM: deKlerkVallentin} to the formulation~\cref{EQ:LMOM} in the following way.
\begin{proposition} \label{prop:L0 standard form}
Consider an instance of~\cref{EQ:LMOM} with feasible solution $L_0 \in \R[\mx]_{2t}^*$ that satisfies the property: $L_0(g_i p^2) = 0$ implies $g_i p^2 = \sum_{j=1}^\ell p_j h_j$ for some polynomials $p_j$ with $\deg(p_j) \leq 2t-\deg(h_j)$. 
Let $X_0 = M_t(L_0) \oplus \big(\bigoplus_{i \in [m]} M_t(g_iL_0)\big)$ be the matrix associated % M \succeq 0$ and $M(g_i) \succeq 0$,~$i \in [m]$ be the matrices associated 
to $L_0$ via \cref{EQ:LM} and \cref{EQ:LMg}. If all non-zero eigenvalues of $X_0$ are at least $r>0$, then $B(X_0,r/2)$ is contained in the feasible region of \cref{eq:MOMSDP}.
\end{proposition}
\begin{proof}
Let $X \in B(X_0,r/2)$ and write $X = X_0 + \tilde X$. Here $B(X_0,r/2)$ is the ball of radius $r/2$ (in the Frobenius norm) in the affine space $V(\mathcal F)$ defined by the linear equalities of \cref{eq:MOMSDP}.  We first show % To show that $X$ is feasible for \cref{eq:MOMSDP} it thus suffices to show 
that a zero-eigenvector of $X_0$ is a zero-eigenvector of $X$. Since $X_0$ is block-diagonal, a zero-eigenvector of $X_0$ corresponds to a zero-eigenvector $v$ of a single block $M_t(g_i L_0)$. For such a vector $v$ we have %If $v$ is a zero-eigenvector of $M(g_i)$, then $
\[
0=v^T M_t(g_i L_0) v = L_0(g_i(\mx) p_v(\mx)^2)
\]
where $p_v(\mx) = \sum_{|\alpha| \leq t-\lceil \deg(g_i)/2\rceil} v_\alpha \mx^\alpha$. By assumption on $L_0$, we therefore have 
$%\[
g_i p_v^2 = \sum_{j=1}^\ell p_j h_j
$ %\]
for polynomials $p_j$ with $\deg(p_j) \leq 2t-\deg(h_j)$. 

To show that $v$ corresponds to a zero-eigenvector of $X$ it remains to observe that, since $X \in V(\mathcal F)$, %the linear constraints on $X$ are such that 
% The linear constraints $\langle A_i, X \rangle = b_i$ ($i \in [K]$) are such that 
it corresponds to a linear functional $L \in \R\cx_{2t}^*$ via \cref{EQ:LM} and \cref{EQ:LMg} with the property that $L(h_j p) = 0$ for all $j \in [m]$ and polynomials $p$ of degree at most $2t-\deg(h_j)$. 
In particular, 
\[
L(g_i(\mx) p_v(\mx)^2) = \sum_{j \in [\ell]} L(h_j p_j) = 0,
\]
and therefore $v$ corresponds to a zero-eigenvector of the $g_i$-th block of $X$. 

Finally, to show that $X$ is positive semidefinite, note that by assumption $X = X_0 + \tilde X$ where $\|\tilde X\|_F \leq r/2$. In particular, the eigenvalues of $\tilde X$ are all at most $r/2$. Since the kernel of $X_0$ is contained in the kernel of $X$, this means that all non-zero eigenvalues of $X$ are at least $r/2$ and hence $X$ is positive semidefinite. 
\end{proof}

As a corollary of \cref{prop:L0 standard form} we obtain \cref{PROP:main}.

\begin{proof}[Proof of~\cref{PROP:main}]
By assumption $L$ is such that for all $g \in \mg$ and $p \in \R\cx_{t-\lceil \deg(g)/2\rceil}$ we have $L(gp^2)=0$ implies $gp^2 = \sum_{j=1}^\ell p_j h_j$ for some polynomials $p_j$ with $\deg(p_j) \leq 2t - \deg(h_j)$. Moreover, all non-zero eigenvalues of $M_t(L) \oplus \big(\bigoplus_{i \in [m]} M_t(g_i L)\big)$ are at least $2^{-\poly(n)}$. By \cref{prop:L0 standard form}, the feasible region of \cref{eq:MOMSDP} contains a ball of radius $2^{-\poly(n)}$ (in $V(\mathcal{F})$), and by \cref{LEM:Lupbound} it is contained in a ball of radius $2^{\poly(n)}$ . \cref{THM: deKlerkVallentin} thus shows that we can compute an $\eps$-additive approximation to the value $\momt$ in time polynomial in $n$ and $\log(1/\eps)$. 
\end{proof}

\section{An algebraic condition for polynomial-time computability: Proof of Theorem~\ref{THM:SUFFICIENT}}

We split the proof of \cref{THM:SUFFICIENT} into two parts: we first consider the moment matrix $M_t(L)$ and then the localizing matrices $M_t(gL)$ for $g \in \mg$. A key tool is the following auxiliary lemma.
\begin{lemma}[Lemma 3.1 in \cite{RaghavendraWeitz:bitcomplexity}.] \label{LEM:integermatrix}
Let $M \in \Z^{N \times N}$ be a symmetric integer matrix with $|M_{ij}| \leq B$ for all $i, j \in [N]$. Then each
non-zero eigenvalue of M has absolute value at least $(BN)^{-N}$.
\end{lemma}
Using this lemma, we are able to show the following.
\begin{proposition} \label{PROP:momentmatrixeigs}
Assume that $\Sgh \subseteq B(0,R)$ is explicitely bounded with $R \leq 2^{\poly(n)}$, and let $L \in \R[\mathbf{x}]^*_{2t}$ be a feasible solution to~\cref{EQ:LMOM}.
Assume further that $L(\mathbf{x}^\alpha) \in \Q$ has polynomial bit-complexity for all $\alpha \in \N^n_{2t}$.
Then $\lambda_{\min}(M_t(L))$, the smallest non-zero eigenvalue of the moment matrix $M_t(L)$ of~\cref{EQ:LM}, satisfies:
\[
    \lambda_{\min}(M_t(L)) \geq 2^{-\poly(n)}.
\]
\end{proposition}
\begin{proof}
By assumption, there exists an integer $C \leq 2^{\poly(n)}$ such that $C \cdot M_t(L)$ is an integer matrix. Furthermore, using \cref{LEM:Lupbound}, the largest entry of $C \cdot M_t(L)$ in absolute value is bounded from above by $C \cdot R^{2t} \leq 2^{\poly(n)}$.
Recall that the matrix $M_t(L)$ is indexed by the monomials of degree at most~$t$, and is thus of size $h(n,t) \coloneqq\binom{n + t }{t}$. As $t \in \N$ is fixed, we have $h(n,t) \leq \poly(n)$.
We may therefore use \cref{LEM:integermatrix} to conclude that the smallest non-zero eigenvalue of $M_t(L)$ is at least
\begin{equation} \label{EQ:momentmatrixeigs}
      \lambda_{\min}(M_t(L)) \geq (2^{\poly(n)} \cdot h(n,t))^{-h(n,t)} \geq 2^{-\poly(n)},
\end{equation}
as required.

\end{proof}
\begin{proposition} \label{PROP:localizingmatrixeigs}
Under the same assumptions as in \cref{PROP:momentmatrixeigs}, let $g \in \R\cx$ be one of the constraints defining $S(\mathbf{g}, \mathbf{h})$. Then the smallest non-zero eigenvalue of the localizing matrix $M_t(gL)$ of~\cref{EQ:LMg} satisfies 
% \[
    $\lambda_{\min}(M_t(gL)) \geq 2^{-\poly(n)}$.
% \]
\end{proposition}
\begin{proof}
We may express $g$ in the monomial basis as:
\[
    g(\mathbf{x}) = \sum_{|\alpha| \leq d} g_\alpha \mathbf{x}^\alpha \quad (g_\alpha \in \Q). 
\]
The coefficients $g_\alpha$ have polynomial bit-complexity by assumption, and so $|g_\alpha| \leq 2^{\poly(n)}$ for each $|\alpha| \leq d$. For the same reason, there exists an integer $C_g \leq 2^{\poly(n)}$ such that $C_g g_\alpha \in \Z$ for all $|\alpha| \leq d$. Recall that the entries of $M_t(gL)$ are linear combinations of the entries of $M_t(L)$, namely for $|\alpha|,|\beta| \leq t -\lceil\deg(g_i)/2\rceil$ the $(\alpha,\beta)$-th entry is of the form:
\[
    M_t(gL)_{\alpha, \beta} = \sum_{|\gamma| \leq \deg(g_i)}  g_\gamma L(\mx^{\alpha+\beta + \gamma}). % \quad (|\alpha + \beta| \leq 2t).
\]
Now, as in  \cref{PROP:momentmatrixeigs}, we see that $C_g \cdot C \cdot M_t(gL)$ is an integer matrix for some integer $C \leq 2^{\poly(n)}$. Furthermore, since the entries of $M_t(L)$ are at most $2^{\poly(n)}$, the entries of $M_t(gL)$ are bounded from above by 
\[
    2^{\poly(n)} \cdot \sum_{|\alpha| \leq d} |g_\alpha| \leq 2^{\poly(n)} \cdot \binom{n + d}{d} \cdot C_g \leq 2^{\poly(n)}.
\]
As before, we may thus invoke \cref{LEM:integermatrix} to conclude the proof.
\end{proof}

\section{A geometric condition for polynomial-time computability: Proof of Theorem~\ref{THM:FULLDIM}}
\label{SEC:GEOM}
Recall that we consider in \cref{THM:FULLDIM} an explicitly bounded semialgebraic set $S(\mathbf{g}) \subseteq \R^n$ with the additional geometric assumption:
\begin{equation} \label{EQ:Ballinclusion}
    B(z, r) \subseteq S(\mathbf{g})
\end{equation}  
for some $r \geq 2^{-\poly(n)}$ and $z \in \R^n$. To prove \cref{THM:FULLDIM}, we will exploit this assumption to exhibit a feasible solution $L \in \R[\mx]_{2t}^*$ to~\cref{EQ:LMOM} that satisfies the conditions of~\cref{THM:SUFFICIENT}.

We begin by noting that the inclusion \cref{EQ:Ballinclusion} implies that $S(\mg)$ contains a translated hypercube $B := [-r, r]^n + z$, for a (slightly smaller) $r \geq 2^{-\poly(n)}$.
We consider the probability measure $\mu_z$ obtained by restricting the Lebesgue measure to $B$ and renormalizing. We show that the operator $L_{\mu_z} \in \R[\mx]_{2t}^*$ associated to $\mu_z$ via~\cref{EQ:measureoperator} satisfies the conditions of \cref{PROP:main}. Let us first note that condition~(1) is satisfied automatically as $B$ is full-dimensional. Indeed, this means that $\int_{S(\mathbf{g})} g p^2 d\mu = 0$ if and only if $gp^2 = 0$. It remains to show that condition~(2) also holds, for which we use \cref{THM:SUFFICIENT}.

For simplicity, we assume first that $z = 0$, so that $B = [-r, r]^n$. In this case, we may use an explicit formula for the moments of $\mu_0$ (see, e.g., \cite{deKlerkLaurent:survey}) to find:
\begin{equation*} 
     L_{\mu_0}(\mx^\alpha) = \int_{[-r,r]^n} \mx^\alpha d\mu_0(\mx) = \begin{cases} \prod_{i=1}^n \frac{r^{\alpha_i}}{\alpha_i + 1} \, &\text{if } \alpha \in (2\N)^n, \\ 0 &\text{otherwise}. \end{cases}
\end{equation*}
As $r \geq 2^{-\poly(n)}$, it follows that:
\begin{equation} \label{EQ:mu0bc}
    \mathrm{bitcomplexity}(L_{\mu_0}(\mx^\alpha)) \leq \poly(n, t)
\end{equation}
for all $|\alpha| \leq 2t$. As $t$ is fixed, $L_{\mu_0}(\mx^\alpha)$ thus has polynomial bit complexity in $n$. That is, $L_{\mu_0}$ satisfies the condition of \cref{THM:SUFFICIENT}.

It remains to consider the case $z \neq 0$. Since $S(\mathbf{g})$ is explicitly bounded, we must have $\|z\|_2 \leq R \leq 2^{\poly(n)}$. After possibly choosing a slightly smaller $r$, we may assume that $z \in \Q^n$ and that~$z$ has polynomial bit-complexity in $n$. 
First, note for all $|\alpha| \leq 2t$ that
\begin{align*}
   L_{\mu_z}(\mx^\alpha) &= \int_{[-r,r]^n + z} \mx^\alpha d\mu_z(\mx) = \int_{[-r,r]^n} (\mx - z)^\alpha d\mu_0(\mx) \\&= L_{\mu_0}((\mx-z)^\alpha).
\end{align*}
Second, note that the coefficients $c_\gamma$ in the expansion of $(\mx - z)^\alpha$ in the monomial basis all have polynomial bit-complexity. The nonzero coefficients are those for which $\gamma$ is entrywise less than or equal to $\alpha$, which we denote $\gamma \leq \alpha$. Note that there are at most $h(n, t) := {n+t \choose t}$ such coefficients. We may thus use~\cref{EQ:mu0bc} to find that 
\begin{equation} \label{EQ:muzbc}
  L_{\mu_z}(\mx^\alpha) = L_{\mu_0}((\mx-z)^\alpha) = \sum_{\gamma \leq \alpha} c_\gamma \cdot L_{\mu_0}(\mx^\gamma)
\end{equation}
has polynomial bit-complexity for all $\alpha \in \N^n$ with $|\alpha| \leq 2t$ (as $t$ is fixed). It follows that $L_{\mu_z}$ satisfies the condition of \cref{THM:SUFFICIENT}. 
%TODO: Above we have quasipoly many terms, so quasipoly bound!

\section{From moments to sums of squares: proof of Theorem~\ref{THM:CONNECTION}}

\cref{PROP:main,THM:FULLDIM} show that, under their respective conditions, we can find an additive $\eps$-approximation to $\mathrm{mom}(f)_t  = \mathrm{sos}(f)_t$ in time $\poly(n,\log(1/\eps))$ by solving \cref{eq:MOMSDP}. We now show that, under these conditions, we have a compact sum-of-squares proof for this bound as well.
\begin{theorem}[Detailed version of \cref{THM:CONNECTION}] \label{THM:CONNECTION-formal}
Let $S(\mathbf{g}, \mathbf{h})$ be a semialgebraic set and suppose that the conditions of Theorem~\ref{PROP:main} or Theorem~\ref{THM:FULLDIM} are satisfied for some fixed $t \geq \lceil \deg(f)/2\rceil$. Suppose that $f -\lambda$ has a sum-of-squares decomposition:
\begin{equation} \label{eq:sos decomp}
    f(\mx) -\lambda = \sigma_0(\mx) + \sum_{i=1}^m g_i(\mx) \sigma_i(\mx)  + \sum_{j=1}^\ell h_j(\mx) p_j(\mx),
\end{equation}
where $\deg(h_j p_j) \leq 2t$ and the sum-of-squares are of the form $\sigma_i = \sum_{k=1}^{K} s_{i,k}^2$ for %polynomials 
$s_{i,k} \in \R\cx$ with $\deg(g_i s_{i,k}^2) \leq 2t$.
Then, for a fixed $\eps>0$, there exists a polynomial $\err(\mx)$ of degree $2t$ such that 
\begin{equation} \label{eq:rounded cert}
    f(\mx) -\lambda +\err(\mx) = \tilde \sigma_0(\mx) + \sum_{i=1}^m g_i(\mx) \tilde \sigma_i(\mx)  + \sum_{j=1}^\ell h_j(\mx) \tilde p_j(\mx) 
\end{equation}
where, for $i = 0,\ldots,m$, we have $\tilde \sigma_i(\mx) = \sum_{k=1}^K \tilde s_{i,k}^2$ for polynomials $\tilde s_{i,k}$ with bit-complexity $\poly(n,\log(1/\eps))$ and $\|\err\|_1 \leq \eps R^{2t} 2^{\poly(n)}$. In particular, this proves nonnegativity of $f(\mx) -\lambda + R^{2t} \|\err\|_1$ on $\Sgh$.
\end{theorem}
From the SDP-formulation of \cref{EQ:introductionSOS} it follows that one can bound $K$ by the rank of the matrices involved, i.e., we have $K \leq \binom{n+t}{t}$ and thus $K \leq \poly(n)$.

\begin{proof}[Proof of \cref{THM:CONNECTION-formal}]
Let $\sigma_i = \sum_{k=1}^K s_{i,k}^2$ and $p_j$ as in \cref{eq:sos decomp} be given. We first show that the coefficients of the $s_{i,k}$ are upper bounded. 
To do so, let $L$ be a linear functional that satisfies the %\mu$ be a Borel probability measure supported on $\Sgh$ that satisfies the 
conditions of \cref{PROP:main} (in the proof of \cref{THM:FULLDIM} we also construct such an $L$). 
% Let $L_\mu \in \R\cx^*$ be the associated linear functional: 
% \[
% L_\mu(p) = \int_{\Sgh} p d\mu(\mx) \qquad \text{for all $p \in \R\cx$.}
% \]
%We first show that the $2$-norm of the coefficients of the $s_{i,k}$ is upper bounded. 
We then have 
%If $L$ is the feasible linear functional we constructed above, we have:
%\[
   $ L(f-\lambda) = L(\sigma_0) + \sum_{i=1}^m L(g_i \sigma_i)$, and since all terms on the right-hand side are nonnegative, this implies in particular that (using Lemma~\ref{LEM:Lupbound}):
\begin{align*} 
    L(g_i s_{i,k}^2) &\leq L(f-\lambda) \leq R^{2t} \cdot \mathrm{poly}(\mathrm{bitcomplexity}(f)),
\end{align*}
for all $0 \leq i \leq m$ and $1 \leq k \leq K$, 
where we set $g_0 = 1$ for convenience. We can now distinguish two cases: (i) $L(g_i s_{i,k}^2) \neq 0$, or (ii) $L(g_i s_{i,k}^2) =0$. In the first case, if $L(g_i s_{i,k}^2) \neq 0$, then we also have 
\[
    L(g_i s_{i,k}^2) \geq \lambda_{\min}(M_t(g_i L)) \cdot \| s_{i,k} \|_2^2,
\]
where $\lambda_{\min}(M_t(g_i L))$ is the smallest non-zero eigenvalue of $M_t(g_i L)$,
and so
\[
\| s_{i,k} \|_2^2 \leq R^{2t} \cdot \lambda_{\min}(M_t( g_i L))^{-1} \cdot \mathrm{poly}(\mathrm{bitcomplexity}(f)).
\]
For the second case, if on the other hand $L(g_i s_{i,k}^2) = 0$, then by condition (1) of \cref{PROP:main} we have $g_i s_{i,k}^2 = \sum_{j =1}^{\ell} h_j q_j$ for some polynomials $q_j$. We may thus remove such $s_{i,k}$ from the sum-of-squares part of \cref{eq:sos decomp} and add them to the ideal part of the certificate. 

Using the above bound on the coefficients of the $s_{i,k}$, we now show a bound on the size of the coefficients of the $p_j$. The polynomial identity \cref{eq:sos decomp} allows us to view the coefficients of the $p_j$ as the solution of a linear system $A \mathbf p = \mathbf b$ where $\mathbf p$ is a vector that contains the coefficients of the $p_j$ ($\mathbf p$ is at most $\ell \binom{n+2t}{2t}$-dimensional), $A$ contains coefficients of the $h_j$, and $\mathbf b$ is the $\binom{n+2t}{2t}$-dimensional vector that contains the coefficients of $f(\mx)-\lambda - \sum_{i=0}^m g_i(\mx) \sigma_i(\mx)$. The system $A \mathbf p = \mathbf b$ is feasible, $\mathbf b \neq 0$, and therefore $r = \rank(A)$ is strictly positive. Let $\overline A$ be an invertible $r$-by-$r$ submatrix of $A$ and write $\overline{\mathbf p}$ and $\overline{\mathbf b}$ for the restrictions of $\mathbf p$ and $\mathbf b$ to the corresponding rows/columns. Cramer's rule then shows that the $i$th coordinate of $\overline{\mathbf p}$ can be written as 
\[
\overline{\mathbf p}_i = \frac{\det(\overline A_i)}{\det(\overline A)}
\]
where $\overline A_i$ is the matrix formed by replacing the $i$th column of $\overline A$ with the vector $\mathbf \overline b$. To upper bound $|\overline{\mathbf p}_i|$ we must give a lower bound on $|\det(\overline A)|$ and an upper bound on $|\det(\overline A_i)|$.
Let us first observe that $|\det(\overline A)| \geq 2^{-\poly(n)}$. 
Indeed, $\overline A$ is an invertible $r$-by-$r$ matrix with $r \in \poly(n)$, its entries have bit-complexity $\poly(n)$ since they correspond to coefficients of the $h_j$, and therefore applying \cref{LEM:integermatrix} to a suitable integer multiple of $\overline A$ shows %implies 
that $|\det(\overline A)| \geq 2^{-\poly(n)}$. To upper bound $|\det(\overline A_i)|$ it suffices to observe that all entries of $\overline A_i$ are upper bounded in absolute value by $2^{\poly(n)}$: for the $i$th column this follows from the above-derived bound on the $\sigma_i$, for the other columns, as before, we observe that they contain coefficients of the $h_j$. 
By setting all remaining coordinates to zero, we can extend $\overline{\mathbf p}$ to a feasible solution $\mathbf p$ of $A \mathbf p = \mathbf b$. 
To summarize, this shows that there exists a sum-of-squares decomposition of $f-\lambda$ as in \cref{eq:sos decomp} where $\|s_{i,k}\|_2^2 \leq R^{2t} \lambda_{\min}(M_t( g_i L))^{-1} \cdot \mathrm{poly}(n)$ for all $i,k$ and $\|p_j\|_\infty \leq 2^{\poly(n)}$ for all $j$.

We finally show that rounding each coefficient of this certificate to few bits introduces a small error. For each $i \in [m], j \in [K]$, let $\tilde s_{i,k}$ be the polynomial $s_{i,k}$ with each of the coefficients rounded to the nearest integer multiple of $\eps$. Therefore, for $|\alpha|\leq t$ we have $|(s_{i,k})_{\alpha} - (\tilde s_{i,k})_{\alpha}| \leq \eps$.  Using the identity 
% \[
$\tilde s_{i,k}^2 - s_{i,k}^2 = {(\tilde s_{i,k} + s_{i,k}) (\tilde s_{i,k}-s_{i,k})}$,
% \]
this shows $|(\tilde s_{i,k}^2 - s_{i,k}^2)_\alpha| \leq \eps (2(s_{i,k})_{\alpha}+\eps)$ and thus ${\|\tilde s_{i,k}^2 - s_{i,k}^2\|_1} \leq \eps(2\|s_{i,k}\|_1 + \eps\binom{n+2t}{2t})$. Similarly, for each $j \in [\ell]$, let $\tilde p_j$ be the polynomial $p_j$ with each of the coefficients rounded to the nearest integer multiple of $\eps$. Then we have \cref{eq:rounded cert} for the polynomial $\err(\mx)$ defined as 
\[
\err(\mx) := \sum_{i=0}^m g_i (\tilde \sigma_i - \sigma_i) + \sum_{j=1}^\ell h_j (\tilde p_j - p_j)
\]
and hence  
\begin{align*}
\|\err\|_1 &\leq \sum_{i=0} \|g_i\|_1 \left(\sum_{k=1}^K \|\tilde s_{i,k}^2 - s_{i,k}^2\|_1 \right) + \sum_{j=0}^\ell \|h_j\|_1 \|\tilde p_j - p_j\|_\infty.
\end{align*}
As shown above, we have $K \in \poly(n)$, $\|\tilde p_j-p_j\|_\infty \leq \eps$, and 
\[
\|\tilde s_{i,k}^2 - s_{i,k}^2\|_1 \leq \eps\left( R^{2t} \lambda_{\min}(M_t(g_i L))^{-1} \poly(n) + \eps\binom{n+2t}{2t}\right).
\]
Using \cref{PROP:localizingmatrixeigs} we moreover have $\lambda_{\min}(M_t(g_i L))^{-1} \leq 2^{\poly(n)}$. Combining these estimates shows $\|\err\|_1 \leq \eps R^{2t} 2^{\poly(n)}$.

The statement $f(\mx) - \lambda + R^{2t}\|\err\|_1 \geq 0$ on $\Sgh$ follows from~\cref{eq:rounded cert} by adding $R^{2t}\|\err\|_1 - \err(\mx)$ to both sides of the equation, using the fact that 
$
    R^{2t}\|\err\|_1 - \err(\mx) \in \mathcal M(\mathbf g)_{2t},
$
which follows from \cref{lem:powers in M} below.
\end{proof}

\begin{lemma} \label{lem:powers in M}
Let $g_1,\ldots,g_m \in \R\cx$ be such that for $R>0$ we have $R^2-x_i^2 \in \mathcal M(\mathbf g)_2$ for each $i \in [n]$. For $t \in \N$ and $\alpha \in \N^n$ with $|\alpha|\leq 2t$, we have
$
    R^{|\alpha|} %+(|\alpha| \, \mathrm{mod}\, 2)}
    -x^\alpha \in \mathcal M(\mathbf g)_{2t}.
$
\end{lemma}
\begin{proof}
We first consider the case of coordinatewise even exponents, i.e., $R^{2|\alpha|}-x^{2\alpha}$. For this we use induction on~$|\alpha|$. The case $|\alpha|=1$ holds by assumption and for $|\alpha|>1$ we use the identity
\[
R^{2|\alpha|} - x^{2\alpha} = R^{2|\alpha|-2}(R^2 -x_j^2) + x_j^2(R^{2|\alpha|-2} -x^{2(\alpha-e_j)}) \in \mathcal M(\mathbf g)_{2|\alpha|}
\]
where $j \in [n]$ is an index for which $\alpha_j>0$ and $e_j$ is the $j$-th unit vector. Here we use that the first term on the right hand side belongs to $\mathcal M(\mathbf g)_2$ and the second term belongs to $\mathcal M(\mathbf g)_{2|\alpha|}$ as $R^{2|\alpha|-2} -x^{2(\alpha-e_j)} \in \mathcal M(\mathbf g)_{2|\alpha|-2}$ by the induction hypothesis. 

Now let $\gamma \in \N^n$ with $|\gamma|\leq 2t$. We distinguish two cases: $|\gamma|$ is odd or even. When $|\gamma|$ is odd we write $\gamma = \alpha + \beta$ for $\alpha, \beta \in \N^n$ with $|\alpha|,|\beta|\leq t$ and $|\alpha|+1 = |\beta| = \frac{|\gamma|+1}{2}$. We then observe that we have the following identity
\begin{align*}
\frac{(Rx^\alpha - x^\beta)^2 + R^2(R^{2|\alpha|} - x^{2\alpha}) + (R^{2|\beta|} - x^{2\beta})}{2R} \\= \frac{R^{2+2|\alpha|} + R^{2|\beta|} - 2R x^{\alpha+\beta}}{2R} =R^{|\gamma|} - x^{\gamma} 
\end{align*}
where in the last equality we use the identities $R^{|\gamma|} = R^{|\alpha|+|\beta|} = R^{1+2|\alpha|} = R^{2|\beta|-1}$. % which hold by our assumptions on $|\alpha|,|\beta|$. 
In the first part of the proof we have shown that $R^{2|\alpha|}-x^{2 \alpha}, R^{2|\beta|}-x^{2\beta} \in \mathcal M(\mathbf g)_{2t}$ and since moreover $(Rx^{\alpha}-x^{\beta})^2 \in \mathcal M(\mathbf g)_{2t}$, the above identity thus shows that $R^{|\gamma|}-x^\gamma \in \mathcal M(\mathbf g)_{2t}$.
Finally, for the case where $|\gamma|$ is even we use a similar argument. We write $\gamma = \alpha + \beta$ for $\alpha,\beta \in \N^n$ with $|\alpha|=|\beta|$ and use the identity 
\[
R^{|\gamma|} - x^\gamma = \frac{(x^\alpha-x^\beta)^2 + (R^{2\alpha}-x^{2\alpha}) + (R^{2\beta}-x^{2\beta})}{2}. \qedhere
\]
\end{proof}

\section{Discussion}
We have given algebraic and geometric conditions that guarantee polynomial-time computability of the moment-SOS hierarchy for polynomial optimization problems~\cref{EQ:POP}. In the general, explicitly bounded setting, our conditions are similar to the ones considered by Raghavendra \& Weitz~\cite{RaghavendraWeitz:bitcomplexity} to show existence of compact sum-of-squares certificates. For full-dimensional feasible regions $S(\mg)$, we give explicit, geometric conditions, which include for instance that $S(\mg)$ either contains a small ball, a strictly feasible point of low bit-complexity, or has sufficient volume. Furthermore, we make explicit the connection between polynomial-time computability of the bound $\momt$ and the existence of compact feasible solutions to the sum-of-squares formulation $\sost$, which is only implicitly present in~\cite{RaghavendraWeitz:bitcomplexity}.

\subsection*{A general geometric condition} \cref{THM:FULLDIM} applies only when the feasible region $S(\mg)$ of~\cref{EQ:POP} is a full-dimensional semialgebraic set. It would be very interesting to formulate a similar, geometric condition that guarantees polynomial-time computability of the moment-SOS hierarchy in the general case. This requires finding an appropriate analog of the second condition of~\cref{THM:FULLDIM} in the setting where $S(\mg, \mh)$ might not be full-dimensional.

\subsection*{Relation to the complexity of SDP} Our present discussion relates closely to the more general study of the computational complexity of semidefinite programming. It is an open question whether SDPs can be solved to (near-)optimality in polynomial-time. Even the exact complexity of testing feasibility of SDPs is not known. We do know that in the bit-model, membership of the feasibility problem in NP and Co-NP is simultaneous~\cite{Ramana:sdp}. (In the real number model of Blum-Shub-Smale it lies in ${\text{NP} \cap \text{Co-NP}}$~\cite{Ramana:sdp}.) 

On the positive side, polynomial-time solvability of SDPs is guaranteed when the feasible region contains an `inner ball' and is contained in an `outer ball' of appropriate size~\cite{schrijveretal:ellipsoid, deKlerkVallentin:SDPcomplexity}. On the negative side, there are several classes of relatively simple SDPs whose feasible solutions nonetheless have exponential bit-complexity, see, e.g.,~\cite{Pataki:exponentialsize} and the discussion therein.

In principle, these positive and negative results on SDPs provide conditions on \cref{EQ:introductionSOS} and \cref{EQ:LMOM} that (partly) show when polynomial-time computation can and cannot be guaranteed. 
% In principle, one could use these results to (partly) map when polynomial-time computation of \cref{EQ:introductionSOS} and \cref{EQ:LMOM} can and cannot be guaranteed. 
The key difference with our results is that we only impose conditions on the original polynomial optimization problem~\cref{EQ:POP}, rather than on the semidefinite programs resulting from the moment/sum-of-squares relaxations.

\subsection*{Finding exact SOS-decompositions}

In the setting of polynomial optimization, it usually suffices to find \emph{approximate} SOS-decompositions, for which one can use (standard) SDP-solvers. The problem of finding exact SOS-decompositions is more complicated. In the general case one could in principle use, for example, quantifier-elimination algorithms~\cite{Basu:algorithms,Renegar:quantifier}. In the univariate case specialized algorithms have been developed %However, to find exact SOS-decompositions of univariate polynomials alternative methods exist %such To We finally mention some prior work on the complexity of testing whether a given polynomial has a sum-of-squares decomposition. In the univariate setting several algorithms exist which are not based on solving SDPs~
\cite{Schweighofer:univariatealgorithm,Chevillardetal:univariateapproximation}, see also~\cite{Magronetal:univariatesos}. We note however that none of these methods come with polynomial runtime guarantees. 

\subsection*{Acknowledgments}
We thank Harold Nieuwboer for bringing to our attention the literature relating to the volume of neighborhoods of algebraic sets. We further thank the anonymous ISSAC referees for their helpful comments and suggestions. In particular, for their suggestion to also consider the regime $t = O(\log n)$ (see \cref{APP:log}).

\bibliographystyle{plain}
\bibliography{surveybib}

\appendix
\section{A quasipolynomial runtime guarantee when $t = O(\log n)$} \label{APP:log}
Throughout, we have assumed that the level $t$ of the hierarchy is \emph{fixed}. As the SDPs \eqref{EQ:introductionSOS} and \eqref{EQ:LMOM} involve matrices of size $h(n,t) = {n+t \choose t} \lesssim n^t$, this is a natural assumption if one is interested in polynomial-time computation. In the regime $t \approx \log n$, one obtains instead SDPs of \emph{quasi}polynomial size $h(n,t) \approx n^{\poly(\log n)}$, which are still potentially useful. For instance, the SOS-hierarchy has been applied in this regime to obtain algorithms with sub-exponential runtime for certain problems in robust estimation~\cite{KSS:robustestimation}.

Our main results carry over to this setting in a natural way, showing quasipolynomial runtime guarantees under the same assumptions on $S(\mathbf g, \mathbf h)$ and $f$ but with $t = O(\log n)$. For ease of reference, we restate them here in this form.
\begin{proposition}[cf. \cref{PROP:main}] \label{PROP:mainquasi}
Let $S(\mathbf{g}, \mathbf{h})$ be a semialgebraic set and let $t = O(\log n)$ with $t \geq \lceil \deg(f)/2\rceil$. Assume that $S(\mathbf{g}, \mathbf{h})$ is explicitly bounded: $g_1(\mx) = R^2 - \sum_{i=1}^n \mx_i^2$ for some $1 \leq R \leq 2^{\mathrm{poly}(n)}$. Suppose furthermore that there exists an $L \in \R[\mx]_{2t}^*$ with $L(1)=1$ and the following properties:
\begin{enumerate}
    \item For any $g \in \mathbf g$ and any  $p \in \R\cx_{t-\lceil \deg(g)/2\rceil}$, if $L(g p^2) =0$, then there are $p_1, p_2, \ldots, p_\ell\in\R[\mx]$ such that:
\[
	g p^2(\mx) = \sum_{j=1}^\ell p_j(\mx) h_j(\mx),
\]
and $\deg(p_j) \leq 2t-\deg(h_j)$ for each $j \in [\ell]$. We recall that $1 \in \mg$ by convention.
\item The matrices $M_t(L)$ and $M_t(gL)$ ($g \in \mathbf g$) have smallest non-zero eigenvalue at least $2^{-{\mathrm{quasipoly}(n)}}$. 
\end{enumerate}
Then for $\varepsilon \geq 2^{-\mathrm{poly}(n)}$, the bound $\mathrm{mom}(f)_t$ (which equals $\mathrm{sos}(f)_t$) may be computed in quasipolynomial time in $n$ up to an additive error of at most $\varepsilon$.
\end{proposition}
\begin{proof}
    The arguments of \cref{sec:SDPformulation} apply analogously.
\end{proof}
\begin{theorem}[cf. \cref{THM:SUFFICIENT}] \label{THM:SUFFICIENTquasi}
    Let $\Sgh$ be an explicitly bounded semialgebraic set with $R \leq 2^{\poly(n)}$ and let $L \in \R[\mx]_{2t}^*$ be a feasible solution to $\momt$ for $t = O(\log n)$. Assume that $L(\mx^{\alpha}) \in \Q$ has quasipolynomial bit-complexity for all $\alpha \in \N^n_{2t}$. Then the smallest non-zero eigenvalue of $M_t(L)$ is at least $2^{-\mathrm{quasipoly}(n)}$ and the same holds for the localizing matrices $M_t(gL)$ for $g \in \mathbf g$.
\end{theorem}
\begin{proof}
    \cref{PROP:momentmatrixeigs} and \cref{PROP:localizingmatrixeigs}, which make up the proof of \cref{THM:SUFFICIENT}, generalize to the setting $t = O(\log n)$ straightforwardly. The key difference is that, if $t = O(\log n)$, the number of monomials $h(n, t)$ of degree at most $t$ is larger than $\mathrm{poly}(n)$ (but still quasipolynomial in $n$). As a result, the inequality~\eqref{EQ:momentmatrixeigs} only yields a bound $\lambda_{\min}(M_t(L)) \geq 2^{-\mathrm{quasipoly}(n)}$.
\end{proof}

\begin{theorem}[cf. \cref{THM:FULLDIM}] \label{THM:FULLDIMquasi}
Let $S(\mathbf{g}) \subseteq \R^n$ be a semialgebraic set defined only by inequalities. Assume that the following two conditions are satisfied:
\begin{enumerate}
    \item $S(\mathbf{g})$ is explicitly bounded: $g_1(\mx) = R^2 - \sum_{i=1}^n \mx_i^2$ with constant $1 \leq R \leq 2^{\mathrm{poly}(n)}$.
    \item $S(\mathbf{g})$ contains a ball of radius $r \geq 2^{-\mathrm{poly}(n)}$, i.e., $B(z, r) \subseteq S(\mathbf{g})$ for some $z \in \R^n$.
\end{enumerate}
Then, for $t = O(\log n)$ with $t \geq \lceil \deg(f)/2\rceil$ and $\varepsilon \geq 2^{-\mathrm{poly}(n)}$, the bound $\mathrm{sos}(f)_t$ may be computed in quasipolynomial time in~$n$ up to an additive error of at most $\varepsilon$. 
\end{theorem}
\begin{proof}
    The arguments of \cref{SEC:GEOM} generalize directly to the setting where $t = O(\log n)$, see in particular \cref{EQ:mu0bc} and \cref{EQ:muzbc}. This yields a feasible solution $L \in \R[\mx]_{2t}$ to~\cref{EQ:LMOM} for which $L(\mx^\alpha)$ has quasipolynomial bitcomplexity in $n$ for all $|\alpha| \leq 2t$. Therefore, we can use \cref{THM:SUFFICIENTquasi} to bound the smallest (non-zero) eigenvalues of the (localizing) moment matrices of $L$. Finally, we may apply \cref{PROP:mainquasi} to finish the proof.
\end{proof}

\clearpage 

\end{document}